\documentclass[12pt]{amsart}
\usepackage{enumerate}
\usepackage{geometry}
\usepackage{amsmath,amssymb,amsthm,mathrsfs,amsfonts}
\usepackage{graphicx}
\usepackage{tikz}
\usetikzlibrary{shapes}
\usetikzlibrary{plotmarks}
\usetikzlibrary{arrows}
\usetikzlibrary{positioning}
\usepackage{tkz-graph}
\makeatletter
\@namedef{subjclassname@2010}{%
  \textup{2010} Mathematics Subject Classification}
\makeatother
\newtheorem{thm}{Theorem}[section]

\newtheorem{lem}[thm]{Lemma}
\newtheorem{prop}[thm]{Proposition}

\newtheorem{mainthm}[thm]{Main Theorem}

\theoremstyle{definition}
\newtheorem{defn}[thm]{Definition}
\newtheorem{rem}[thm]{Remark}

\numberwithin{equation}{section}

\newcommand\myeq{\mathrel{\stackrel{\makebox[0pt]{\mbox{\normalfont\tiny def}}}{=}}}

\def\a#1{\mathfrak{#1}}

\def\R{\mathrm{RRA}}

\def\Me{\mathrm{REL}}
\def\con#1{#1 \breve{\text{  }}}

\def\i{1\text{\textquoteright}}

\begin{document}


\baselineskip=17pt



\title[Stone type theorems via games]{Stone type representation theorems via games}

\author[T. Aslan]{Tu\u{g}ba Aslan}
\address{Tu\u{g}ba Aslan, Faculty of Engineering and Natural Sciences, Bahcesehir University, Istanbul, Turkey}
\email{tgbasln@gmail.com}
\author[M. Khaled]{Mohamed Khaled}
\address{Mohamed Khaled, Faculty of Engineering and Natural Sciences, Bahcesehir University, Istanbul, Turkey \& Alfr\'ed R\'enyi Institute of Mathematics, Hungarian Academy of Sciences, Budapest, Hungary}
\email{mohamed.khalifa@eng.bau.edu.tr}

\date{}

\begin{abstract}
The classes of relativized relation algebras (whose units are not necessarily transitive as binary relations) are known to be finitely axiomatizable. In this article, we give a new proof for this fact that is easier and more transparent than the original proofs. We give direct constructions for all cases, whereas the original proofs reduced the problem to only one case. The proof herein is combinatorial and it uses some techniques from game theory.
\end{abstract}

\subjclass[2010]{Primary 03G15; Secondary 06E25, 03B45}

\keywords{representability, relation algebras, games and networks}
\maketitle

\section{Introduction}
Relation algebras are algebras of binary relations that were introduced by A. Tarski in 1940's (see, e.g. \cite{relalg}). The creation of these algebras was heralded by the pioneering work of A. De Morgan, E. Schr\"oder and C. S. Peirce in the theory of binary relations. Relation algebras were shown to be important in various disciplines, e.g. in mathematics, computer science, linguistics and cognitive science. These algebras were heavily studied by many scholars in the fields of logic and algebra alike.

Here, we consider some varieties containing the ``concrete'' relation algebras, namely the relativized relation set algebras. Let $H\subseteq\{r,s\}$, where $r$ and $s$ stand for reflexive and symmetric respectively. A relation $W\subseteq U\times U$ is said to be an $H$-relation on $U$ if it satisfies the properties in $H$. For instance, if $r\in H$ then we expect $W$ to contain $(u,u)$, for each $u\in U$.

\begin{defn}\label{def}
A relativized relation set algebra $\a{A}$ is a subalgebra of an algebra of the form $$\a{Re}(W)\myeq\langle \mathcal{P}(W), \cup,\cap,\setminus,\emptyset,W,\circ,\otimes,\delta\rangle,$$
where $W\subseteq U\times U$, for some set $U$, and the non-Boolean operations are defined as follows: $\delta=\{(x,y)\in W: x=y\}$. Let $R,S\subseteq W$, then
$$R\circ S=\{(x,y)\in W: \exists z\in U ((x,z)\in R\text{ and }(z,y)\in S)\}$$ and $\otimes R=\{(x,y)\in W:(y,x)\in R\}$. The set $W$ is called the unit of $\a{A}$ and the smallest set $U$ that satisfies $W\subseteq U\times U$ is called the base of $\a{A}$. The class of all relativized relation set algebras is denoted by $\R_{\emptyset}$. Let $H\subseteq\{r,s\}$. The class of $H$-relativized relation set algebras is given by
$$\R_H=\{\a{A}\in \R_{\emptyset}: \text{ the unit of }\a{A}\text{ is an } H-\text{relation on the base of }\a{A}\}.$$
\end{defn}
The classes of relativized relation algebras were shown to be finitely axiomatizable. For the special case $H=\{r,s\}$, the finite axiomatizability was established by R. Maddux
\cite{madd82} in 1982. Then, in 1991, R. Kramer \cite{kramer} proved the finite axiomatizability for arbitrary $H$, by reducing the problem to the case $H=\{r,s\}$ and then applying Maddux's result \cite{madd82}. A different axiomatization for the case $H=\{r,s\}$ was also given by M. Marx et al. \cite{marx}.

In \cite[Theorem 7.5]{HHB}, R. Hirsch and I. Hodkinson used a new technique to simplify Maddux's proof. They used games and networks to build step by step representations. Despite the fact that this method was proved to be effective, it was not used to simplify Kramer's proof yet. Note that \cite[Theorem 7.5]{HHB} can not be applied verbatim to the case of arbitrary $H$, as it essentially depends on the reflexivity and the symmetry of the  units of the algebras in $\R_{\{r,s\}}$.

In this paper, we give a new proof for the finite axiomatizability of the class $\R_H$ for any arbitrary $H\subseteq\{r,s\}$. We give a direct construction for all the cases, unlike the original proof \cite{kramer}. We adapt Kramer's axioms and we generalize the method of Hirsch and Hodkinson in a non-trivial way.
\begin{defn}[R. Kramer \cite{kramer}]\label{axioms}
The class $\Me$ is defined to be the class of all algebras of the form
$$\a{A}=\langle A,+,\cdot,-,0,1,;,\con{}, \i\rangle\footnote{Here, the operations $+,\cdot,-,0,1,;,\con{}$ and $\i$ are the abstract versions of the concrete operations $\cup,\cap,\setminus,\emptyset, W, \circ,\otimes$ and $\delta$, respectively, in  Definition~\ref{def}.}$$
which satisfy the axioms (Ax 1) through (Ax 9) listed below.
\begin{enumerate}[(Ax 9)]
\item[(Ax 1)] $\a{BlA}=\langle A,+,\cdot,-,0,1\rangle$ is a Boolean algebra.
\item[(Ax 2)] $\con{(x\cdot\con{y})}=\con{x}\cdot y$.
\item[(Ax 3)] $(x+y);z=x;z+y;z$ and $x;(y+z)=x;y+x;z$.
\item[(Ax 4)] $1;0=0$ and $0;1=0$.
\item[(Ax 5)] $(\con{x};y)\cdot z=(\con{x};(y\cdot (\con{(\con{x})};z)))\cdot z$ and $(x;\con{y})\cdot z=((x\cdot(z;\con{(\con{y})}));\con{y})\cdot z$.
\item[(Ax 6)] $\i;x\leq x$ and $x;\i\leq x$.
\item[(Ax 7)] $\i;\i=\i$.
\item[(Ax 8)] $(-\con{1};-\con{1})\cdot \i=0$.
\item[(Ax 9)] $((x\cdot \i);y);z=(x\cdot \i);(y;z)$, $(x;(y\cdot \i));z=x;((y\cdot \i);z)$ and $(x;y);(z\cdot \i)=x;(y;(z\cdot \i))$.
\end{enumerate}
Let $H\subseteq\{r,s\}$ be arbitrary. We define $\Me_H$ to be the class that consists of those $\a{A}\in\Me$ that satisfy (Ax $p$) given below, for each property $p\in H$. Note that $\Me=\Me_{\emptyset}$.
\begin{enumerate}[(Ax $p$)]
\item[(Ax $s$)] $\con{1}=1$.
\item[(Ax $r$)] $\i;1=1$ and $1;\i=1$.
\end{enumerate}
\end{defn}
\begin{rem}
We use the convention that $-\con{x}=-(\con{x})$. For example, axiom (Ax 8) above says that $(-(\con{1});-(\con{1}))\cdot \i=0$.
\end{rem}
For any class $\mathrm{K}$ of algebras, $\mathbb{I}\mathrm{K}$ is the class that consists of all isomorphic copies of the members of $\mathrm{K}$. As we mentioned before, we aim to reprove the following theorem.
\begin{mainthm}\label{main}
For each $H\subseteq\{r,s\}$, we have $\Me_H=\mathbb{I}\R_H$.
\end{mainthm}

\section{Atoms in the perfect extensions}
Recall the basic concepts of Boolean algebras with operators (BAO) from the literature, see e.g. \cite{tarski}. For any Boolean algebra with operators $\mathfrak{B}$, let $At(\mathfrak{B})$ be the set of all atoms in $\mathfrak{B}$.

Let $\a{A}\in\Me$. Denote by $\a{A}^+$ the perfect extension of $\a{A}$ defined in \cite{tarski} (to  form this perfect extension, we need to make sure that both operators $;$ and $\con{}$ are normal and additive, but this follows from axioms (Ax 3), (Ax 4) and \cite[Theorem 1.3 (i), (iii)]{kramer}). Clearly, $\a{A}^+\in\Me$ since the negation occurs only in a constant axiom and in the Boolean axioms, see \cite[Theorem 2.18]{tarski}. If one can show that $\a{A}^+$ is representable, i.e. $\a{A}^+\in\mathbb{I}\R_{\emptyset}$, then we can deduce that $\a{A}$ is also representable.

These perfect extensions are complete, atomic and their non-Boolean operations are completely additive \cite[Theorem 2.4 and Theorem 2.15]{tarski}. So, we define perfect algebras as follows.

\begin{defn}\label{perfect}
An algebra $\a{A}\in\Me$ is said to be a perfect algebra if and only if $\a{A}$ is complete, atomic and its conversion and composition are completely additive.
\end{defn}
Now, we prove several Lemmas that seem to be interesting in their own right. To represent a perfect algebra $\a{A}$, we roughly represent each atom $a\in At(\a{A})$ by a tuple $(s,e)$. Geometrically, such a tuple $(s,e)$ can be viewed as an arrow that starts at $s$ and ends at $e$. Then, we show that $\a{A}$ can be embedded into the full algebra whose unit consists of all arrows representing atoms.

The real challenge now is to arrange that
the unit has the desired properties, by adding the converse $(e,s)$, the  starting $(s,s)$, or the ending $(e,e)$ arrows whenever it is necessary. Such new arrows need to be associated to some atoms to keep the claim that each atom is represented by some arrows (maybe more than one), and there are no irrelevant arrows. For example, the following Lemma defines the converse of an ``arrow-atom'', if it exists.

\begin{lem}\label{cdef}
Let $\a{A}\in\Me$ be a perfect algebra, and let $a,b\in At(\a{A})$ be such that $\con{a}\not=0$ and $\con{b}\not=0$. Then the following are true.
\begin{enumerate}[(1)]
\item $\con{a}$ is an atom in $\a{A}$.
\item $\con{(\con{a})}\not=0$ and $\con{(\con{a})}=a$.
\item $\con{a}=\con{b}\implies a=b$.
\end{enumerate}
\end{lem}
\begin{proof}Let $\a{A}\in\Me$ be a perfect algebra, and let $a,b\in At(\a{A})$ be such that $\con{a}\not=0$ and $\con{b}\not=0$. Then, by \cite[Theorem 1.3 (i), (iv) and (v)]{kramer}, $a\cdot\con{1}\not=0$.
\begin{enumerate}[(1)]
\item The atomicity and the completeness of $\a{A}$ imply that $$1=\sum\{a^{-}\in\a{A}: a^{-}\text{ is an atom}\}.$$
Since $\con{}$ is completely additive, there exists an atom $a^{-}\in At(\a{A})$ such that $a\cdot\con{(a^{-})}\not=0$. By \cite[Theorem 1.3 (ii)]{kramer}, we must also have $\con{a}\cdot a^{-}\not=0$. Thus, $a^{-}\leq\con{a}$ and $a\leq\con{(a^{-})}$. Hence, by \cite[Theorem 1.3 (iii), (iv)]{kramer},
$$a^{-}\leq \con{a}\leq \con{(\con{(a^{-})})}\leq a^{-}.$$
Hence, $a^{-}=\con{a}$ which means that $\con{a}$ is an atom.
\item We have shown that $\con{a}$ is an atom. Thus, by \cite[Theorem 1.3 (v)]{kramer}, we have $\con{(\con{(\con{a})})}\not=0$. By the fact that $\con{}$ is a normal operator, it follows that $\con{(\con{a})}\not=0$. By \cite[Theorem 1.3 (iv)]{kramer}, we have $\con{(\con{a})}\leq a$. But $a$ is an atom, thus $\con{(\con{a})}=a$ as desired.
\item This follows immediately from (2), indeed $$\con{a}=\con{b}\implies a=\con{(\con{a})}=\con{(\con{b})}=b.$$
\end{enumerate}
Thus, we have shown that the conversion operator $\con{}$ is an idempotent bijection if it is restricted to the set $\{a\in At(\a{A}):\con{a}\not=0\}$.
\end{proof}
Similarly, we need to define identity atoms for the starting and the ending arrows of each atom, if there are such identity atoms. To test whether any of the starting arrow or the ending arrow exists for an atom $a$, we define the following entities: $st(a)=\i;a$ and $end(a)=a;\i$.

\begin{lem}\label{sedef}Let $\a{A}\in\Me$ be a perfect algebra, and let $a\in At(\a{A})$.
\begin{enumerate}[(1)]
\item Suppose $st(a)\not=0$. Then there is a unique atom $a^{-}\in At(\a{A})$ (denoted by $\mathcal{S}a$) such that $a^{-}\leq\i$ and $a^{-};a=a$.
\item Suppose $end(a)\not=0$. Then there is a unique atom $a^{-}\in At(\a{A})$ (denoted by $\mathcal{E}a$) such that $a^{-}\leq\i$ and $a;a^{-}=a$.
\end{enumerate}
\end{lem}
\begin{proof}
Let $\a{A}\in\Me$ be a perfect algebra, and let $a\in At(\a{A})$ be an atom.
\begin{enumerate}[(1)]
\item Suppose that $st(a)=\i;a\not=0$. The existence of $\mathcal{S}a$ is guaranteed by the perfectness of $\a{A}$ (and axiom (Ax 6)). For the uniqueness, suppose that there are two atoms $a_1^{-}$ and $a_2^{-}$ such that $a^{-}_1\leq \i$, $a^{-}_2\leq \i$, $a_1^{-};a=a$ and $a_2^{-};a=a$. Then,
\begin{eqnarray*}
a&\leq &(a_1^{-};1)\cdot (a_2^{-};1)\hspace{0.8cm} \text{by axiom (Ax 3)}\\
&=& a_1^{-};(a_2^{-};1) \hspace{1cm} \text{by \cite[Theorem 1.8 (iii)]{kramer}}\\
&=& (a_1^{-};a_2^{-});1 \hspace{1cm} \text{by axiom (Ax 9)}\\
&=& (a_1^{-}\cdot a_2^{-});1 \hspace{1cm} \text{by \cite[Theorem 1.8 (i)]{kramer}}.
\end{eqnarray*}
Thus, by axiom (Ax 4) and the fact that both $a^{-}_1$ and $a^{-}_2$ are atoms, it follows that $a_1^{-}=a_2^{-}$.
\item The proof is similar to the proof of item (1) above: use axiom (Ax 9) and \cite[Theorem 1.8 (i) and (iv)]{kramer}.\qedhere
\end{enumerate}
\end{proof}
The following Lemma shows that the processes $\mathcal{S}$ and $\mathcal{E}$ of defining the identity atoms are compatible with the conversion of atoms.

\begin{figure}[!ht]
\begin{tikzpicture}
\tikzset{edge/.style = {->,> = latex'}}
\node (a) at  (0,0) {$s$};
\node (b) at  (3,0) {$e$};
\draw[edge] (a) to [bend left,above] node {$a$} (b);
\draw[edge,dashed] (b) to [bend left,above] node {$\con{a}$} (a);
\Loop[dist=0.7cm,dir=EA,labelstyle=below right,style=dashed](b);
\Loop[dist=0.7cm,dir=WE,labelstyle=below left,style=dashed](a);
\node (c) at  (-0.5,-0.7) {$\mathcal{S}a$};
\node (d) at  (3.5,-0.7) {$\mathcal{E}a$};
\end{tikzpicture}
\end{figure}

\begin{lem}\label{secomp}Let $\a{A}\in\Me$ be a perfect algebra, and let $a\in At(\a{A})$ be an atom such that $\con{a}\not=0$.
\begin{enumerate}[(1)]
\item $st(a)\not=0\implies end(\con{a})\not=0\text{ and }\mathcal{E}\con{a}=\mathcal{S}a$.
\item $end(a)\not=0\implies st(\con{a})\not=0\text{ and }\mathcal{S}\con{a}=\mathcal{E}a$.
\item $st(a)\not=0\implies \mathcal{S}a\leq a;\con{a}$.
\item $end(a)\not=0\implies \mathcal{E}a\leq \con{a};a$.
\end{enumerate}
\end{lem}
\begin{proof}
Let $\a{A}\in\Me$ be a perfect algebra, and let $a\in At(\a{A})$ be an atom such that $\con{a}\not=0$.
\begin{enumerate}[(1)]
\item Suppose that $st(a)\not=0$. Recall that $a= \mathcal{S}a;a$. Then, 
\begin{eqnarray*}
\con{a}&=&\con{(\mathcal{S}a;a)}\\
&=&\con{(\con{(\con{(\mathcal{S}a)})};\con{(\con{a})})}\hspace{0.45cm} \text{by Lemma~\ref{cdef} (2) and \cite[Theorem 1.8 (ii)]{kramer}}\\
&=&\con{(\con{(\con{a};\con{(\mathcal{S}a)})})}\hspace{0.65cm} \text{by \cite[Theorem 1.8 (v)]{kramer}}\\
&\leq& \con{a};\con{(\mathcal{S}a)}\hspace{1.2cm}\text{by \cite[Theorem 1.3 (iv)]{kramer}}\\
&=&\con{a};\mathcal{S}a \hspace{1.2cm}\text{by \cite[Theorem 1.8 (ii)]{kramer}}
\end{eqnarray*}
On the other hand, (Ax 6) implies $\con{a};\mathcal{S}a\leq\con{a}$. Hence $\con{a};\mathcal{S}a=\con{a}$. Therefore, $end(\con{a})\not=0$ and $\mathcal{E}\con{a}=\mathcal{S}a$ as desired.
\item The proof is similar to the proof of item (1) above: use Lemma~\ref{cdef} (2), \cite[Theorem 1.8 (ii), (v)]{kramer} and \cite[Theorem 1.3 (iv)]{kramer}.
\item We note that axiom (Ax 5), Lemma~\ref{sedef} (1) and Lemma~\ref{cdef} (2) imply $a=(\mathcal{S}a;a)\cdot a=((\mathcal{S}a\cdot(a;\con{a}));a)\cdot a$. Thus, by axiom (Ax 4), we must have $\mathcal{S}a\cdot(a;\con{a})\not=0$. In other words, $\mathcal{S}a\leq a;\con{a}$.
\item The proof is similar to the proof of item (3) above.\qedhere
\end{enumerate}
\end{proof}
\begin{lem}\label{identityatoms}
Let $\a{A}\in\Me$ be a perfect algebra. Let $a\in At(\a{A})$ be an atom such that $a\leq \i$. Then, $\con{a}\not=0$, $st(a)\not=0$, $end(a)\not=0$, and $$a=\con{a}=\mathcal{S}a=\mathcal{E}a.$$
\end{lem}
\begin{proof}
This follows immediately from \cite[Theorem 1.8 (i) and (ii)]{kramer}.
\end{proof}
Now, we show that $\mathcal{S}$ and $\mathcal{E}$ are compatible with each other and with the composition operator in the perfect algebras.
\begin{lem}\label{secons}
Let $\a{A}\in\Me$ be a perfect algebra, and let $a,b,c\in At(\a{A})$ be some atoms. Suppose that $a\leq b;c$, then the following are true.
\begin{enumerate}[(1)]
\item \begin{enumerate}[(i)]
\item $st(a)\not=0\iff st(b)\not=0$.
\item $st(a)\not=0 \text{ and } st(b)\not=0 \implies \mathcal{S}a=\mathcal{S}b$.
\end{enumerate}
\item \begin{enumerate}[(i)]
\item $end(a)\not=0\iff end(c)\not=0$.
\item $end(a)\not=0 \text{ and } end(c)\not=0 \implies \mathcal{E}a=\mathcal{E}c$.
\end{enumerate}
\item \begin{enumerate}[(i)]
\item $end(b)\not=0\iff st(c)\not=0$.
\item $end(b)\not=0 \text{ and } st(c)\not=0 \implies \mathcal{E}b=\mathcal{S}c$.
\end{enumerate}
\end{enumerate}
\end{lem}
\begin{proof}
Let $\a{A}\in\Me$ be a perfect algebra, and let $a,b,c\in At(\a{A})$ be some atoms. Suppose that $a\leq b;c$. We will only show (1) and the other items can be shown in the same way. Suppose $st(a)\not=0$, then $a=\mathcal{S}a;a$. Hence, by axiom (Ax 9),
$$a=\mathcal{S}a;a\leq \mathcal{S}a;(b;c)\leq (\mathcal{S}a;b);c.$$
Thus, $\mathcal{S}a;b\not=0$. Now, by axiom (Ax 6), it is easy to see that $\mathcal{S}a;b=b$, hence $st(b)\not=0$ and $\mathcal{S}b=\mathcal{S}a$. Conversely, suppose that $st(b)\not=0$. Then $b=\mathcal{S}b;b$. Hence, by (Ax 9),
$$a\leq b;c=(\mathcal{S}b;b);c=\mathcal{S}b;(b;c).$$
So by axiom (Ax 5), $0\not=a=(\mathcal{S}b;(b;c))\cdot a=(\mathcal{S}b;((b;c)\cdot(\mathcal{S}b;a)))\cdot a$. Thus, by axioms (Ax 4), we have $st(b)\not=0$ and $(\mathcal{S}b;a)=a$, which means that $st(a)\not=0$ and $\mathcal{S}a=\mathcal{S}b$. Therefore, both (i) and (ii) hold as desired.
\end{proof}
The following Lemma is needed for the constructions in the next section.
\begin{lem}
\label{cycles} Let $\a{A}\in\Me$ be a perfect algebra, and let $a,b,c\in At(\a{A})$ be some atoms. Suppose that $a\leq b;c$, then the following are true.
\begin{enumerate}[(1)]
\item $a\leq\i \implies b=\con{c}\text{ and }c=\con{b}$.
\item $\con{b}\not=0\implies c\leq \con{b};a$.
\item $\con{c}\not=0\implies b\leq a;\con{c}$.
\item $\con{a}\not=0\text{ and }\con{b}\not=0\implies \con{b}\leq c;\con{a}$.
\item $\con{a}\not=0\text{ and }\con{c}\not=0\implies \con{c}\leq \con{a};b$.
\item $\con{a}\not=0, \con{b}\not=0\text{ and }\con{c}\not=0\implies \con{a}\leq \con{c};\con{b}$.
\end{enumerate}
\end{lem}
\begin{figure}[!ht]
\centering
\begin{tikzpicture}
\centering
\tikzset{edge/.style = {->,> = latex'}}
\node[shape=circle] (a) at  (0,0) {};
\node[shape=circle] (b) at  (4,0) {};
\draw[edge] (a)  to (b);
\node[shape=circle] (1) at  (2,-0.2) {$a$};
\node[shape=circle] (c) at  (2,2) {};
\draw[edge] (a)  to (c);
\draw[edge] (c)  to (b);
\node[shape=circle] (2) at  (1.2,1) {$b$};
\node[shape=circle] (3) at  (2.8,1) {$c$};
\draw[edge] (c)  to[bend right] (a);
\node[shape=circle] (4) at  (0.35,1.45) {$\con{b}$};
\node[shape=circle] (a1) at  (6,0) {};
\node[shape=circle] (b1) at  (10,0) {};
\draw[edge] (a1)  to (b1);
\node[shape=circle] (11) at  (8,-0.2) {$a$};
\node[shape=circle] (c1) at  (8,2) {};
\draw[edge] (a1)  to (c1);
\draw[edge] (c1)  to (b1);
\node[shape=circle] (21) at  (7.2,1) {$b$};
\node[shape=circle] (31) at  (8.8,1) {$c$};
\draw[edge] (b1)  to[bend right] (c1);
\node[shape=circle] (41) at  (9.65,1.45) {$\con{c}$};
\end{tikzpicture}
\end{figure}
\begin{proof}
Let $\a{A}\in\Me$ be a perfect algebra, and let $a,b,c\in At(\a{A})$ be such that $a\leq b;c$.
\begin{enumerate}[(1)]
\item Suppose that $a\leq \i$. Then, by \cite[Theorem 1.8 (vi)]{kramer}, we have $\con{b}\not=0$ and $\con{c}\not=0$. By \cite[Theorem 1.8 (vi)]{kramer} and axiom (Ax 5), it follows that
$$0\not=a=a\cdot (b;c)=a\cdot (\con{c};\con{b})=(\con{c};(\con{b}\cdot(\con{(\con{c})};a)))\cdot a.$$
Thus, $\con{b}\cdot(\con{(\con{c})};a)\not=0$. By axiom (Ax 6), we have $\con{(\con{c})};a\leq \con{(\con{c})}$. Then, $\con{b}\cdot\con{(\con{c})}\not=0$. Therefore, \cite[Theorem 1.3 (i) and (vi)]{kramer} implies that $b\cdot\con{c}\not=0$, i.e. $b=\con{c}$. Consequently, by \cite[Theorem 1.3 (ii)]{kramer}, $c\cdot\con{b}\not=0$ and $c=\con{b}$.
\item Suppose that $\con{b}\not=0$. Now, by Lemma~\ref{cdef} (2) and axiom (Ax 5),
$$0\not=a=a\cdot (b;c)= a\cdot (\con{(\con{b})};c)=(\con{(\con{b})};(c\cdot (\con{(\con{(\con{b})})};a)))\cdot a.$$
Thus, by axiom (Ax 4), $c\leq \con{(\con{(\con{b})})};a$. The desired follows by \cite[Theorem 1.3 (iv)]{kramer} (or by Lemma~\ref{cdef} (2)).
\item Similarly to item (2) above using the second part of axiom (Ax 5).
\item Apply item (2) then apply item (3), the desired follows.
\item Apply item (3) then apply item (2), the desired follows.
\item Apply item (4) then apply item (2), the desired follows.\qedhere
\end{enumerate}
\end{proof}
\section{Games and networks}
Throughout this section, fix a perfect algebra $\a{A}\in\Me$. We also use von Neumann ordinals.
\begin{defn}An $\a{A}$-pre-network (pre-network, for short) is a  pair $N= (N_1,N_2)$, where $N_1$ is a (possibly
empty) set, and $N_2:N_1\times N_1\rightarrow At(\a{A})$ is a partial map. We write $nodes(N)$ for $N_1$ and $edges(N)$ for the domain of $N_2$; we also may write $N$ for any of $N$, $N_2$, $nodes(N)$ and $edges(N)$.
\end{defn}
We  write $\emptyset$ for the pre-network $(\emptyset,\emptyset)$. For the pre-networks $N$ and $N'$, we write $N\subseteq N'$ if and only if $nodes(N)\subseteq nodes(N')$, $edges(N)\subseteq edges(N')$, and $N'(x,y)=N(x,y)$ for all $(x,y)\in edges(N)$.

Let $\alpha$ be an ordinal. A sequence of pre-networks $\langle N_{\kappa}:\kappa\in\alpha\rangle$ is said to be a chain if $N_{\kappa_1}\subseteq N_{\kappa_2}$ whenever $\kappa_1\in\kappa_2$. Supposing that $\langle N_{\kappa}:\kappa\in\alpha\rangle$ is a chain of pre-networks, define the pre-network $N=\bigcup\{N_{\kappa}:\kappa\in\alpha\}$ with $nodes(N)=\bigcup\{nodes(N_{\kappa}):\kappa\in\alpha\}$, $edges(N)=\bigcup\{edges(N_{\kappa}):\kappa\in\alpha\}$ and its labeling is given as follows: For each $(x,y)\in edges(N)$, we let $N(x,y)=N_{\kappa}(x,y)$, where $\kappa\in\alpha$ is any ordinal with $(x,y)\in edges(N_{\kappa})$.
\begin{defn} \label{network}
An $\a{A}$-network (network, for short) is a pre-network that satisfies the following:
\begin{enumerate}[(N 1)]
\item For all $(x,y)\in edges(N)$, we have
\begin{enumerate}
\item $N(x,y)\leq \i \iff x= y$.
\item $st(N(x,y))\not=0 \iff (x,x)\in edges(N)$.
\item $end(N(x,y))\not=0 \iff (y,y)\in edges(N)$.
\item $\con{N(x,y)}\not=0 \iff (y,x)\in edges(N)$.
\end{enumerate}
\item For all $(x,y)\in edges(N)$,
\begin{center}if $(y,x)\in edges(N)$ then $N(x,y)\cdot\con{N(y,x)}\not=0$.\end{center}
\item For all $(x,y),(x,z),(z,y)\in edges(N)$, $N(x,y)\cdot(N(x,z);N(z,y))\not=0$.
\end{enumerate}
\end{defn}
The following Lemma defines a network, for each atom $a\in\a{A}$. These networks serve as initial networks, one can get more complex networks by composing these initial networks together.
\begin{lem}\label{restr}
Let $a\in At(\a{A})$ and let $x,y$ be nodes, with $x=y$ iff $a\leq\i$. Let $N \myeq N_{xy}^a$ be the pre-network with nodes $x,y$ and with
the following edges:
\begin{itemize}
\item $(x, y)$ is an edge of $N$, with label $N(x, y) = a$.
\item If $st(a)\not=0$ then $(x, x)$ is an edge of $N$ with label $N(x, x) = \mathcal{S}a$.
\item If $end(a)\not=0$ then $(y, y)$ is an edge of $N$ with label $N(y, y) = \mathcal{E}a$.
\item If $\con{a}\not=0$ then $(y,x)$ is an edge of $N$ with label $N(y,x)=\con{a}$.
\end{itemize}
Then $N$ is a well defined network.
\end{lem}
\begin{proof}
Lemma~\ref{identityatoms} implies that $N$ is a well defined pre-network. Now we need to show that $N$ satisfies all the conditions of Definition~\ref{network}.
\begin{enumerate}[(N 1)]
\item By assumptions, we know $x=y\iff N(x,y)=a\leq\i$. Suppose that $\con{a}\not=0$. By Lemma~\ref{cdef} (2) and \cite[Theorem 1.8 (ii)]{kramer}, we have
$$y=x\iff a\leq\i\iff N(y,x)=\con{a}\leq\i.$$
Thus, Condition (N 1a) holds for $N$. Let us check condition (N 1b).  We know from the construction that
\begin{equation}\label{xy}
st(N(x,y))\not=0\iff st(a)\not=0 \iff (x,x)\in edges(N).
\end{equation}
Suppose that $(y,x)\in edges(N)$. Then, by the construction, $\con{a}\not=0$ and $N(y,x)=\con{a}$. Thus, by Lemma~\ref{cdef} (2) and Lemma~\ref{secomp} (1,2), we have $end(a)\not=0$ if and only if $st(\con{a})\not=0$. Hence,
\begin{eqnarray}
st(N(y,x))\not=0 &\iff& st(\con{a})\not=0 \label{yx}\\
&\iff& end(a)\not=0\nonumber\\
&\iff& (y,y)\in edges(N)\nonumber
\end{eqnarray}
Also, by Lemma~\ref{identityatoms} and \cite[Theorem 1.8 (i)]{kramer}, we must have
\begin{eqnarray}
(x,x)\in edges(N) &\implies& N(x,x)=\mathcal{S}a\leq\i\label{xx}\\
&\implies& st(N(x,x))=\i;\mathcal{S}a=\mathcal{S}a,\nonumber
\end{eqnarray}
and
\begin{eqnarray}
(y,y)\in edges(N) &\implies& N(y,y)=\mathcal{E}a\leq\i\label{yy}\\
&\implies& st(N(y,y))=\mathcal{E}a;\i=\mathcal{E}a.\nonumber
\end{eqnarray}
Therefore, by \eqref{xy}, \eqref{yx}, \eqref{xx} and \eqref{yy}, condition (N 1b) holds for all the edges in $N$ as desired. Similarly, one can check that condition (N 1c) holds for $N$. Also, it is not hard to see that condition (N 1d) follows from the construction, Lemma~\ref{cdef} (2) and Lemma~\ref{identityatoms}.
\item Suppose that $(y,x)\in edges(N)$. That means $\con{a}\not=0$ and $N(y,x)=\con{a}$. Recall that $\con{(\con{a})}=a$, by Lemma~\ref{cdef} (2). Thus,
\begin{eqnarray}
N(x,y)\cdot \con{N(y,x)}&=a\cdot \con{(\con{a})}&=a\not=0,\label{cxy}\\
N(y,x)\cdot \con{N(x,y)}&=\con{a}\cdot \con{a}&=\con{a}\not=0\label{cyx}.
\end{eqnarray}
Moreover, Lemma~\ref{identityatoms} implies that
\begin{eqnarray}
(x,x)\in edges(N) &\implies& N(x,x)=\mathcal{S}a\leq\i\label{cxx}\\
&\implies& N(x,x)\cdot\con{N(x,x)}=\mathcal{S}a\not=0,\nonumber
\end{eqnarray}
and
\begin{eqnarray}
(y,y)\in edges(N) &\implies& N(y,y)=\mathcal{E}a\leq\i\label{cyy}\\
&\implies& N(y,y)\cdot\con{N(y,y)}=\mathcal{E}a\not=0.\nonumber
\end{eqnarray}
Therefore, by \eqref{cxy}, \eqref{cyx}, \eqref{cxx} and \eqref{cyy}, it follows that condition (N 2) holds for the pre-network $N$.
\item Note that, by the construction of $N$, there is no triangle (cycle of length $3$ on three different nodes) in $N$.
\begin{itemize}
\item Suppose that $(x,x)\in edges(N)$ and $N(x,x)=\mathcal{S}a$. So, by the definition of $\mathcal{S}a$, we have $\mathcal{S}a;a=a$. Hence,
\begin{equation}\label{com1}
N(x,y)\cdot(N(x,x);N(x,y))=a\cdot (\mathcal{S}a;a)=a\not=0
\end{equation}
Similarly, if $(y,y)\in edges(N)$ then we must have
\begin{equation}\label{com2}
N(x,y)\cdot(N(x,y);N(y,y))=a\cdot (a;\mathcal{E}a)=a\not=0
\end{equation}
\item Suppose that both $(y,y)$ and $(y,x)$ are edges in $N$. Then $\con{a}\not=0$ and $end(a)\not=0$. Thus, by Lemma~\ref{secomp} (2), we have $st(\con{a})\not=0$ and $\mathcal{S}\con{a}=\mathcal{E}a$. Hence, by Lemma~\ref{sedef}, we have
\begin{eqnarray}
N(y,x)\cdot(N(y,y);N(y,x))&=&\con{a}\cdot (\mathcal{E}a;\con{a})\label{com3}\\
&=&\con{a}\cdot (\mathcal{S}\con{a};\con{a})\nonumber\\
&=&\con{a}\not=0\nonumber.
\end{eqnarray}
Similarly, if $(x,x),(y,x)\in edges(N)$ then $\mathcal{E}\con{a}=\mathcal{S}a$ and
\begin{eqnarray}
N(y,x)\cdot(N(y,x);N(x,x))&=&\con{a}\cdot (\con{a};\mathcal{S}a)\label{com4}\\
&=&\con{a}\cdot (\con{a};\mathcal{E}\con{a})\nonumber\\
&=&\con{a}\not=0\nonumber.
\end{eqnarray}
\item By \cite[Theorem 1.8 (i)]{kramer}, if $(x,x)\in edges(N)$ then we have
\begin{equation}\label{com5}
N(x,x)\cdot(N(x,x);N(x,x))=\mathcal{S}a\cdot(\mathcal{S}a;\mathcal{S}a)=\mathcal{S}a\not=0.
\end{equation}
Similarly, by \cite[Theorem 1.8 (i)]{kramer}, if $(y,y)\in edges(N)$ then
\begin{equation}\label{com6}
N(y,y)\cdot(N(y,y);N(y,y))=\mathcal{E}a\cdot(\mathcal{E}a;\mathcal{E}a)=\mathcal{E}a\not=0.
\end{equation}
\item Now, suppose that $(y,x)\in edges(N)$. If $(x,x)\in edges(N)$ then by Lemma~\ref{secomp} (3) we have
\begin{equation}\label{com7}
N(x,x)\cdot (N(x,y);N(y,x))= \mathcal{S}a\cdot(a;\con{a})=\mathcal{S}a\not=0.
\end{equation}
If $(y,y)\in edges(N)$ then by Lemma~\ref{secomp} (4) we have
\begin{equation}\label{com8}
N(y,y)\cdot (N(y,x);N(x,y))= \mathcal{E}a\cdot(\con{a};a)=\mathcal{E}a\not=0.
\end{equation}
\end{itemize}
The facts in \eqref{com1}, \eqref{com2}, \eqref{com3}, \eqref{com4}, \eqref{com5}, \eqref{com6}, \eqref{com7} and \eqref{com8} imply that $N$ satisfies condition (N 3).
\end{enumerate}
Therefore, we have shown that $N$ is a network as required.
\end{proof}
Now, we are ready to introduce a game between a female $\exists$ and a male $\forall$, then we will show that $\exists$ has a winning strategy. 
\begin{defn}Let $\alpha$ be an ordinal. We define a game, denoted by $G_{\alpha}(\a{A})$, with $\alpha$ rounds, in which the players $\forall$ and $\exists$ build a chain of pre-networks $\langle N_{\kappa}:\kappa\in\alpha\rangle$ as follows. In round $0$, $\exists$ starts by letting $N_0=\emptyset$. Suppose that we are in round $\kappa\in\alpha$ and assume that each $N_{\lambda}$, $\lambda\in\kappa$, is a pre-network. If $\kappa$ is a limit ordinal then $\exists$ defines $N_{\kappa}=\bigcup\{N_{\lambda}:\lambda\in\kappa\}$. If $\kappa=(\kappa-1)+1$ is a successor ordinal then the players move as follows.
\begin{enumerate}[(1)]
\item $\forall$ may choose a non-zero element $a\in \a{A}$, and $\exists$ must respond with
a pre-network $N_{\kappa}\supseteq N_{\kappa-1}$ containing an edge $e$ with $N_{\kappa}(e)\leq a$.
\item Alternatively, $\forall$ may choose an edge $(x,y)\in edges(N_{\kappa-1})$ and two elements $b,c\in \a{A}$ with $N_{\kappa-1}(x,y)\leq b;c$. Then, $\exists$ must respond with a pre-network $N_{\kappa}\supseteq N_{\kappa-1}$ such that for some $z\in N_{\kappa}$ (possibly $z\in N_{\kappa-1}$ already), $(x,z),(z,y)\in N_{\kappa}$, $N_{\kappa}(x,z)\leq b$ and $N_{\kappa}(z,y)\leq c$.
\end{enumerate}
$\exists$ wins if each pre-network $N_{\kappa}$, $\kappa\in\alpha$, played during the game is actually a network. Otherwise, $\forall$ wins. There are no draws.
\end{defn}
\begin{prop}\label{winning} Let $\alpha$ be an ordinal. Then, $\exists$ has a winning strategy in the game $G_{\alpha}(\a{A})$.
\end{prop}
\begin{proof}
Let $\alpha$ be an ordinal and let $\kappa\in \alpha$. By the definition of the game, it is obvious that $\exists$ always wins in the round $\kappa$ if $\kappa=0$ or $\kappa$ is a limit ordinal. So, we may suppose that $\kappa=(\kappa-1)+1$ is a successor ordinal. We also may assume inductively that $\exists$ has managed to guarantee that $N_{\kappa-1}$ is a network. We consider the possible moves that $\forall$ can make.
\begin{enumerate}[(1)]
\item Suppose that $\forall$ picks a non-zero $a\in \a{A}$. Then, $\exists$ chooses $a^{-}\in At(\a{A})$ with $a^{-}\leq a$. She picks brand new nodes $x,y\not\in N_{\kappa-1}$ such that $x=y$ iff $a^{-}\leq \i$. Now, $\exists$ extends $N_{\kappa-1}$ to $N_{\kappa}$ by adding the new nodes $x$ and $y$, and by adding the following edges:
\begin{itemize}
\item She adds $(x,y)$ with label $N_{\kappa}(x,y)=a^{-}$.
\item If $st(a^{-})\not=0$ then she adds $(x,x)$ with label $N_{\kappa}(x,x)=\mathcal{S}a^{-}$.
\item If $end(a^{-})\not=0$ then she adds $(y,y)$ with label $N_{\kappa}(y,y)=\mathcal{E}a^{-}$.
\item If $\con{(a^{-})}\not=0$ then she also adds $(y,x)$ with label $N_{\kappa}(y,x)=\con{(a^{-})}$.
\end{itemize}
Lemma~\ref{identityatoms} implies that $N_{\kappa}$ is a well defined pre-network. Note that there are no edges connecting the old nodes of $N_{\kappa-1}$ together with the new nodes. Also, by the inductive hypothesis, $N_{\kappa-1}$ is a network. Thus, Lemma~\ref{restr} implies that $N_{\kappa}$ is indeed a network.
\item Alternatively, suppose that $\forall$ chooses an edge $(x,y)\in N_{\kappa-1}$ and two elements $b,c\in \a{A}$ with $N_{\kappa-1}(x,y)\leq b;c$. There are two cases. First suppose that there is $z\in N_{\kappa-1}$ such that $(x,z),(z,y)\in N_{\kappa-1}$, $N_{\kappa-1}(x,z)\leq b$ and $N_{\kappa-1}(z,y)\leq c$. In this case, $\exists$ lets $N_{\kappa}=N_{\kappa-1}$.

Now assume there is no such $z$. Since composition is completely additive and $\a{A}$ is atomic, then $$b;c=\sum\{b^{-};c^{-}:b^{-},c^{-}\in At(\a{A}), b^{-}\leq b, c^{-}\leq c\}.$$ So we may choose two atoms $b^{-},c^{-}\in At(\a{A})$ such that $b^{-}\leq b$, $c^{-}\leq c$, and $N_{\kappa-1}(x,y)\leq b^{-};c^{-}$. Our assumptions imply that $b^{-}\leq -\i$ and $c^{-}\leq -\i$. Otherwise, suppose that $b^{-}\leq \i$. Then, $N_{\kappa-1}(x,y)\leq \i;1$, and by condition (N 1b) we must have $(x,x)\in edges(N_{\kappa-1})$. By axiom (Ax 6), it follows that
$$N_{\kappa-1}(x,y)\leq b^{-};c^{-}\implies N_{\kappa-1}(x,y)\leq c^{-} \implies N_{\kappa-1}(x,y)=c^{-}$$
and by axioms (Ax 9) and \cite[Theorem 1.8 (i)]{kramer}, we have
\begin{eqnarray*}
N_{\kappa-1}(x,y)&\leq& N_{\kappa-1}(x,x);N_{\kappa-1}(x,y)\\
&\leq& N_{\kappa-1}(x,x);(b^{-};c^{-})\\
&\leq& (N_{\kappa-1}(x,x);b^{-});c^{-}\\
&\leq& (N_{\kappa-1}(x,x)\cdot b^{-});c^{-}
\end{eqnarray*}
which implies that $N_{\kappa-1}(x,x)\cdot b^{-}\not=0$, in other words $N_{\kappa-1}(x,x)=b^{-}$. Thus, we could choose $z=x$, and this makes a contradiction. So we must have $b^{-}\leq -\i$ and (similarly) $c^{-}\leq -\i$.

Now, $\exists$ chooses a brand new node $z\not\in N_{\kappa-1}$, and she creates $N_{\kappa}$ with nodes those of $N_{\kappa-1}$ plus $z$ and edges those of $N_{\kappa-1}$ plus:
\begin{itemize}
\item $(x,z)$ and $(z,y)$ with labels $N_{\kappa}(x,z)=b^{-}$ and $N_{\kappa}(z,y)=c^{-}$.
\item She adds $(z,x)$ if and only if $\con{(b^{-})}\not=0$, and its label would be $N_{\kappa}(z,x)=\con{(b^{-})}$.
\item She adds $(y,z)$ if and only if $\con{(c^{-})}\not=0$, and its label would be $N_{\kappa}(y,z)=\con{(c^{-})}$.
\item She also adds $(z,z)$ if and only if $end(b^{-})\not=0$, and its label would be $N_{\kappa}(z,z)=\mathcal{E}b^{-}$.
\end{itemize}
Note that $x=y$ if and only if $N_{\kappa-1}(x,y)\leq \i$. Thus, Lemma~\ref{cycles} (1) implies that $N_{\kappa}$ is a well defined pre-network. Similarly to the preceding case, we claim that $N_{\kappa}$ is a network. If
$z\in nodes(N_{\kappa-1})$ then $N_{\kappa}$ is a network by the inductive hypothesis, and we are done. So, suppose that $z$ is not in $N_{\kappa-1}$. In what follows, assume that $N_{\kappa-1}(x,y)=a^{-}$.

For each $V\subseteq nodes(N_{\kappa})$, we define ${N_{\kappa}\upharpoonright}_V$, the {\it restriction} of $N_{\kappa}$ to $V$, to be the pre-network whose nodes are the nodes in $V$ and whose edges are $edges(N_{\kappa})\cap (V\times V)$. Every edge in ${N_{\kappa}\upharpoonright}_V$ keeps its label as in the pre-network $N_{\kappa}$. We note that:
\begin{enumerate}
\item The restriction ${N_{\kappa}\upharpoonright}_{nodes(N_{\kappa-1})}$ is just $N_{\kappa-1}$, which is a network.
\item The restriction ${N_{\kappa}\upharpoonright}_{\{x,z\}}$ is $N_{xz}^{b^{-}}$. This is a consequence of the inductive hypothesis (In particular, the pre-network $N_{\kappa-1}$ satisfies condition (N 1)) and Lemma~\ref{secons} (1).
\item The restriction ${N_{\kappa}\upharpoonright}_{\{z,y\}}$ is $N_{zy}^{c^{-}}$. Again, this follows from the inductive hypothesis and Lemma~\ref{secons} (2,3).
\end{enumerate}
By the inductive hypothesis and Lemma~\ref{restr}, these
three restrictions are networks. It is now apparent that (N 1) and (N 2) hold for $N_{\kappa}$, since for any edge $(p, q)$ of $N_{\kappa}$, both of $p$ and $q$ must lie in one of the three networks above, all of which satisfy (N 1) and (N 2). For the same reason, (N 3) holds whenever the three nodes
mentioned in (N 3) are not pairwise distinct. So, it remains to check (N 3) for every three pairwise distinct nodes $p,q,r$ of $N_{\kappa}$ with $(p,q),(p,r),(r,q)\in edges(N_{\kappa})$.

If $p,q,r$ are all in $N_{\kappa-1}$ then this instance
of (N 3) holds because of the assumption that $N_{\kappa-1}$ is a network. So, we can suppose $z \in\{p, q, r\}$, and
hence that $\{p, q, r\} = \{x, y, z\}$ (since $(p,q),(p,r), (r,q)$ are edges of $N_{\kappa}$). There
are now six cases to consider, corresponding to the six permutations of $\{x, y, z\}$:
\begin{itemize}
\item Remember $a^{-}\leq b^{-};c^{-}$. By the construction, $(x,y)$, $(x,z)$ and $(z,y)$ are all edges in $N_{\kappa}$, and
\begin{equation}\label{fin1}
N_{\kappa}(x,y)\cdot(N_{\kappa}(x,z);N_{\kappa}(z,y))=a^{-}\cdot(b^{-};c^{-})=a^{-}\not=0.
\end{equation}
\item Suppose that $(y,z)\in edges(N_{\kappa})$. Then, we should have $\con{(c^{-})}\not=0$ and $N_{\kappa}(y,z)=\con{(c^{-})}$. Thus, by Lemma~\ref{cycles} (3), it follows that
\begin{equation}\label{fin2}
N_{\kappa}(x,z)\cdot(N_{\kappa}(x,y);N_{\kappa}(y,z))=b^{-}\cdot (a^{-};\con{(c^{-})})=b^{-}\not=0.
\end{equation}
In the same way, by Lemma~\ref{cycles} (2), if $(z,x)\in edges(N_{\kappa})$ then
\begin{equation}\label{fin3}
N_{\kappa}(z,y)\cdot(N_{\kappa}(z,x);N_{\kappa}(x,y))= c^{-}\cdot(\con{(b^{-})};a^{-})=c^{-}\not=0.
\end{equation}
\item Suppose that $(z,x)\in edges(N_{\kappa})$. Then, by the construction, we have $\con{(b^{-})}\not=0$ and $N_{\kappa}(z,x)=\con{(b^{-})}$. Remember $N_{\kappa-1}$ is a network, so if $(y,x)\in edges(N_{\kappa})$ (this happens iff $(y,x)\in edges(N_{\kappa-1}))$ then (N 2) implies that $\con{(a^{-})}\not=0$ and $N_{\kappa}(y,x)=N_{\kappa-1}(y,x)=\con{(a^{-})}$. Hence, by Lemma~\ref{cycles} (4),
\begin{equation}\label{fin4}
N_{\kappa}(z,x)\cdot(N_{\kappa}(z,y);N_{\kappa}(y,x))=\con{(b^{-})}\cdot (c^{-};\con{(a^{-})})=\con{(b^{-})}\not=0.
\end{equation}
Again, by Lemma~\ref{cycles} (5), if $(y,z),(y,x)\in edges(N_{\kappa})$ then
\begin{equation}\label{fin5}
N_{\kappa}(y,z)\cdot(N_{\kappa}(y,x);N_{\kappa}(x,z))=\con{(c^{-})}\cdot (\con{(a^{-})};b^{-})=\con{(c^{-})}\not=0.
\end{equation}
\item Suppose that $(z,x),(y,z),(y,x)$ are all edges in the pre-network $N_{\kappa}$. Then, $\con{(a^{-})}\not=0$, $\con{(b^{-})}\not=0$, $\con{(c^{-})}\not=0$, $N_{\kappa}(y,x)=\con{(a^{-})}$, $N_{\kappa}(z,x)=\con{(b^{-})}$ and $N_{\kappa}(y,z)=\con{(c^{-})}$. Thus, by Lemma~\ref{cycles} (6),
\begin{equation}\label{fin6}
N_{\kappa}(y,x)\cdot(N_{\kappa}(y,z);N_{\kappa}(z,x))=\con{(a^{-})}\cdot(\con{(c^{-})};\con{(b^{-})})=\con{(a^{-})}\not=0.
\end{equation}
\end{itemize}
Hence, by \eqref{fin1}, \eqref{fin2}, \eqref{fin3}, \eqref{fin4}, \eqref{fin5} and \eqref{fin6}, condition (N 3) holds for $N_{\kappa}$ as desired.
\end{enumerate}
Therefore, if $\exists$ plays according to the strategy above she can win any play of the game $G_{\alpha}(\a{A})$, regardless of what moves $\forall$ makes.
\end{proof}
\begin{proof}[Proof of Theorem~\ref{main}]Let $H\subseteq\{r,s\}$. We prove the non-trivial direction. Let $\a{A}\in\Me_H$, then we need to prove that $\a{A}\in \mathbb{I}\R_H$. Let $\a{A}^+\supseteq\a{A}$ be the perfect extension of $\a{A}$ as a Boolean algebra with operators defined in \cite{tarski}. Clearly, $\a{A}^+\in\Me_H$ and $\a{A}^+$ is a perfect algebra in the sense of Definition~\ref{perfect}, see \cite[Theorem 2.4, Theorem 2.15 and Theorem 2.18]{tarski}.

Let $\alpha$ be an ordinal (large enough) and consider a play $\langle N_{\kappa}:\kappa\in\alpha\rangle$ of $G_{\alpha}(\a{A}^+)$ in which $\exists$ plays as in Proposition~\ref{winning}, and $\forall$ plays every possible move at some stage of play. That means,
\begin{enumerate}[(G 1)]
\item each non-zero $a\in \a{A^+}$ is played by $\forall$ in some
round, and
\item for every $b,c\in\a{A}^+$, every $\kappa\in\alpha$, and each
pair $x, y$ of nodes of $N_{\kappa}$ with $(x,y)\in edges(N_{\kappa})$ and $N_{\kappa}(x, y)\leq b;c$, $\forall$ plays $x$, $y$, $b$, $c$ in some round.
\end{enumerate}
This can be guaranteed by choosing $\alpha$ large enough.

Let $U=\bigcup\{nodes(N_{\kappa}):\kappa\in\alpha\}$ and let $$W=\bigcup\{edges(N_{\kappa}):\kappa\in\alpha\}\subseteq U\times U.$$
We need to check that $W$ is an $H$-relation on $U$. If $s\in H$ then, by axiom (Ax $s$) and the fact that $\con{}$ is completely additive, it follows that $\con{a}\not=0$ for every atom $a\in\a{A}^{+}$. Thus, by condition (N 1d), $W$ in this case must be symmetric. Similarly, if $r\in H$ then axiom (Ax {$r$}) plus the complete additivity of the composition give $st(a)\not=0$ and $end(a)\not=0$, for every atom $a\in\a{A}^{+}$. Thus, by conditions (N 1b) and (N 1c), $W$ must be reflexive. Thus, it remains to show that $\a{A}^+$ is  embeddable into $\a{Re}(W)$. For this, we define the following function: For each $a\in \a{A}^+$, let
$$h(a)=\{(x,y)\in W: \exists\kappa\in\alpha \ ((x,y)\in edges(N_{\kappa}) \text{ and }N_{\kappa}(x,y)\leq a)\}.$$
It is not hard to see that $h$ is a Boolean homomorphism. Also, (G 1) above implies that $h$ is one-to-one. We check conversion, composition, and the identity. Let $b,c\in\a{A}^+$. Then, for each $(x,y)\in W$, we have
\begin{eqnarray*}
&& (x,y)\in h(b;c)\\
&\iff& \exists\kappa\in\alpha\left( (x,y)\in N_{\kappa} \text{ and }N_{\kappa}(x,y)\leq b;c\right)\\
&\iff& \exists z\in U \ \exists\kappa\in\alpha \ ((x,z),(z,y)\in N_{\kappa}, N_{\kappa}(x,z)\leq b \text{ and }N_{\kappa}(z,y)\leq c)\\
&\iff& \exists z\in U \left( (x,z)\in h(b)\text{ and }(z,y)\in h(c)\right)\\
&\iff& (x,y)\in h(b)\circ h(c).
\end{eqnarray*}
The second $\iff$ follows by (G 2). For conversion, let $(x,y)\in W$. By conditions (N 1d) and (N 2) of the networks and \cite[Theorem 1.3 (iii) and (iv)]{kramer}, we have
\begin{eqnarray*}
&&(x,y)\in h(\con{b})\\
&\implies& \exists\kappa\in\alpha \ \left( (x,y)\in N_{\kappa}\text{ and }N_{\kappa}(x,y)\leq\con{b}\right)\\
&\implies& \exists\kappa\in\alpha \ \left( (x,y),(y,x)\in N_{\kappa} \text{ and }N_{\kappa}(y,x)\leq \con{(N_{\kappa}(x,y))}\leq \con{(\con{b})}\leq b\right)\\
&\implies& (x,y),(y,x)\in W \text{ and } (y,x)\in h(b) \\ &\implies& (x,y)\in \otimes h(b).
\end{eqnarray*}
Conversely, by condition (N 2) and \cite[Theorem 1.3 (iii)]{kramer},
\begin{eqnarray*}
&&(x,y)\in \otimes h(b) \\
&\implies & (x,y),(y,x)\in W \text{ and }(y,x)\in h(b)\\
&\implies& \exists\kappa\in \alpha \ \left( (x,y),(y,x)\in N_{\kappa}\text{ and }N_{\kappa}(y,x)\leq b\right)\\
&\implies& \exists\kappa\in \alpha \ \left( (x,y),(y,x)\in N_{\kappa}\text{ and }N_{\kappa}(x,y)\leq \con{(N_{\kappa}(y,x))}\leq \con{b}\right)\\
&\implies& (x,y)\in h(\con{b}).
\end{eqnarray*}
Finally, condition (N 1a) guarantees that $h(\i)=\delta=\{(x,y)\in W:x=y\}$. Thus, $\a{A}^{+}\subseteq\a{Re}(W)$. Therefore, $\a{A}\in\mathbb{I}\R_H$ (i.e. $\a{A}$ is representable) as desired.
\end{proof}

\end{document}